\documentclass[11pt]{amsart}
\usepackage{graphicx}
\usepackage{amsmath,amsthm}
\usepackage{amssymb}
\usepackage{pinlabel}
\usepackage[left = 3.5 cm, right = 3.5  cm , top = 3 cm , bottom = 3cm ]{geometry}
\usepackage[utf8]{inputenc}
\usepackage{todonotes}

\newtheorem{thm}{Theorem}[section]
\newtheorem{prop}[thm]{Proposition}
\newtheorem{notation}[thm]{Notation}
\newtheorem{lem}[thm]{Lemma}
\newtheorem{cor}[thm]{Corollary}

\newtheorem{question}[thm]{Question}

\newcommand{\Text}[1]{\text{\textnormal{#1}}}

\theoremstyle{remark}
\newtheorem{remark}[thm]{Remark}

\newtheorem{ex}[thm]{Example}

\theoremstyle{definition}
\newtheorem{definition}[thm]{Definition}

\begin{document}

\title{Pseudo-Anosov maps with small stretch factors on punctured surfaces}
\author{Mehdi Yazdi}
\date{}
\thanks{Partially supported by NSF Grants DMS-1006553 and DMS-1607374.}

\maketitle

\begin{abstract}
Consider the problem of estimating the minimum entropy of pseudo-Anosov maps on a surface of genus $g$ with $n$ punctures. We determine the behaviour of this minimum number for a certain large subset of the $(g,n)$ plane, up to a multiplicative constant. In particular it has been shown that for fixed $n$, this minimum value behaves as $\frac{1}{g}$, proving what Penner speculated in 1991. 
\end{abstract}

\section{Introduction}
Let $S = S_{g,n}$ be the surface of genus $g$ with $n$ punctures. The \emph{mapping class group} of $S$, $\Text{Mod}(S_{g,n})$, is the group of orientation-preserving homeomorphisms of $S$ up to isotopy, where the punctures are preserved set-wise. By the Nielsen-Thurston classification \cite{nielsen1942abbildungsklassen, nielsen1944surface,  thurston1988geometry}, every mapping class has a representative that is either \emph{periodic}, \emph{reducible}, or \emph{pseudo-Anosov}. A homeomorphism $f \colon S \rightarrow S$ is
pseudo-Anosov if there exists a pair of transverse measured singular foliations $(\mathcal{F}^+, \mu^+)$ and $(\mathcal{F}^-, \mu^-)$ and a number $\lambda >1$ such that 
\[f(\mathcal{F}^+, \mu^+) = (\mathcal{F}^+, \lambda \mu^+), \hspace{3mm}f(\mathcal{F}^-, \mu^-) = (\mathcal{F}^-, \lambda^{-1} \mu^-).  \]

The number $\lambda$ is called the \emph{stretch factor} or \emph{dilatation} of $f$, and the foliations $\mathcal{F}^+$ and $\mathcal{F}^-$ are the \emph{unstable} and \emph{stable foliations} respectively. See \cite{casson1988automorphisms} or \cite{fathi1979travaux}\footnote{An English translation by Kim and Margalit is available in \cite{fathi2012thurston}.} for a careful exposition of the Nielsen-Thurston classification. The class of pseudo-Anosov elements is the most interesting class from the point of view of hyperbolic 3-manifolds, since Thurston proved that $f \in \Text{Mod}(S_{g,n})$ is pseudo-Anosov if and only if the mapping torus of $f$ admits a complete hyperbolic metric \cite{thurston1998hyperbolic,otal2001hyperbolization}. Therefore, an intriguing question is to determine the minimum `complexity' between all pseudo-Anosov elements in the mapping class group of $S$. A natural measure of complexity for a topological map is the \emph{topological entropy}. When $f$ is a pseudo-Anosov map, the entropy is equal to $\log(\lambda)$, where $\lambda$ is the stretch factor of $f$ \cite[Proposition 10.13]{fathi1979travaux}. This justifies considering the quantity
\[ l_{g,n} = \min \{ \hspace{1mm} \log(\lambda(f)) \hspace{1mm} |  f \in \Text{Mod}(S_{g,n}) \hspace{2mm} \Text{is pseudo-Anosov} \}, \]
where the minimum exists by \cite{arnoux1981construction}. From a different perspective, $l_{g,n}$ is the length of the shortest geodesic (systole) of the \emph{moduli space} of $S_{g,n}$ equipped with the \emph{Teichm\"uller metric} \cite{fathi1979travaux}. 
Penner initiated the study of $l_{g,n}$ in his seminal work \cite{penner1991bounds}. He showed that there are explicit positive constants $A_1$ and $A_2$ such that for any $g \geq 2$
\[ \frac{A_1}{g} \leq l_{g,0} \leq \frac{A_2}{g}. \]
In other words, $l_{g,0}$ behaves like $\frac{1}{g}$ for $g \geq 2$. 

After Penner, there has been many works aiming to make the constants $A_1, A_2$ more precise \cite{bauer1992upper, hironaka2006family, minakawa2006examples,  hironaka2014small}, to find the exact value of $l_{g,n}$ for small values of $g$ and $n$ \cite{song2002entropies,ham2007minimum, cho2008minimal, lanneau2011minimumbraids, lanneau2011minimum}, or to find the asymptotic of least stretch factor when restricted to certain subgroups or subsets of the mapping class group \cite{song2005upper, farb2008lower, boissy2012pseudo,  hironaka2014penner, agol2016pseudo, loving2019least}. See also \cite{mcmullen2015entropy, liechti2016minimal, malestein2016pseudo, liechti2018minimal, liechti2018nonorientable, yazdi2018lower} for other related research.

Penner speculated that there should be an analogous upper bound in the case $n \neq 0$  \cite{penner1991bounds}. This is known to be true for $1 \leq n \leq 4$ by adjusting Penner's examples of pseudo-Anosov maps \cite{tsai2009asymptotic}. Tasi proposed studying the behaviour of the function $l_{g,n}$ along different rays in the $(g,n)$ plane \cite{tsai2009asymptotic}. With this perspective, Penner's question is the study of $l_{g,n}$ along the ray $n=$ constant. 

At first sight, it might seem to the expert reader that one should be able to use Penner's examples for $S_{g,0}$ to construct examples for $S_{g,n}$, by taking one of the singularities of the stable foliation to be a puncture. The issue is that, in that case one has to assume that roughly $g$ of those singularities are punctures since the singularities are permuted by the map (see Section \ref{Penner-small-dilatation} on Penner's examples). This should make it clear that fixing the number of punctures while having larger genera can not be obtained easily. Our first theorem gives a positive answer to what Penner speculated.

\begin{thm}
\label{thm:main}
For any fixed $n \in \mathbb{N}$, there are positive constants $B_1 = B_1(n)$ and $B_2=B_2(n)$ such that for any $g \geq 2$
\[ \frac{B_1}{g} \leq l_{g,n} \leq \frac{B_2}{g}. \] 
\end{thm}

See Theorem \ref{thm1} for a more quantitative result. Note that the constants $B_1$ and $B_2$ depend on $n$. Therefore we propose the study of $l_{g,n}$ as a function of two variables $g$ and $n$. Valdivia \cite{valdivia2012sequences} proved that for any fixed $r \in \mathbb{Q}^{>0}$, there are positive constants $C_1=C_1(r)$ and $C_2=C_2(r)$ such that for any $(g,n)$ with $g=rn$ we have
\[  \frac{C_1}{g} \leq l_{g,n} \leq \frac{C_2}{g}.  \]
However, it is not even clear whether the constants $C_1$ and $C_2$ depend continuously on $r$ or not. A first guess might be that $l_{g,n}$ is comparable to $\frac{1}{|\chi(S)|}$. However this can not be true, since Tsai \cite{tsai2009asymptotic} proved that for any fixed $g \geq 2$, there are positive constants $D_1=D_1(g)$ and $D_2=D_2(g)$ such that 
\[ D_1 \cdot  \frac{\log(n)}{n} \leq l_{g,n} \leq D_2 \cdot  \frac{\log(n)}{n}.\]
The different behaviour of $l_{g,n}$ along the rays $g= \Text{constant}$ and $n = \Text{constant}$ suggests that the function $l_{g,n}$ might be complicated, even up to multiplicative constants (see \cite{yazdi2018lower}). We determine the behaviour of $l_{g,n}$ on a certain large subset of the $(g,n)$ plane. In particular this region contains balls of arbitrary large radii.

\begin{figure}
\labellist
\pinlabel $g$ at -5 120
\pinlabel $n$ at 120 -5
\pinlabel $2$ at -5 35
\endlabellist
\centering
\includegraphics[width= 1.5 in]{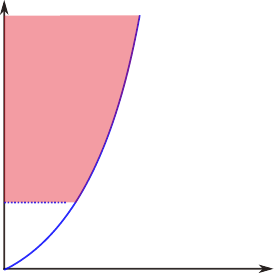}
\caption{The region $ g \geq C \hspace{1mm} n\log^2(n)$}
\label{region}
\end{figure}

\begin{thm}
\label{thm:uniform}
There exist positive constants $A,B$ and $C$ such that for any $n \geq 1$ and $g \geq C \hspace{1mm} n \log^2(n)$ we have
\[ \frac{B}{g} \leq l_{g,n} \leq \frac{A}{g}. \]
\end{thm}

The point of this theorem is that the constants $A,B$ and $C$ depend on neither $g$ nor $n$, in contrast to the previous results. We should note that the lower bounds in Theorems \ref{thm:main} and \ref{thm:uniform} are direct corollaries of Penner's lower bound and our contribution is to provide upper bounds that have the same `order of magnitude' as the existing lower bound. This has been done by constructing examples of pseudo-Anosov maps with `small' stretch factors.

\subsection{Acknowledgement}
Part of this work was carried out while the author was a PhD student at Princeton University. I would like to thank my advisor, Professor David Gabai, for stimulating conversations. The author gratefully acknowledges the support by a Glasstone Research Fellowship.

\section{Background}

\subsection{Perron-Frobenius matrices} \hfill\\
A matrix $A$ is called \textit{Perron-Frobenius} if it has non-negative entries and there is some $k \in \mathbb{N}$ such that all entries of $A^k$ are strictly positive. The Perron-Frobenius theorem states that such a matrix has a unique largest (in absolute value) eigenvalue $\lambda = \rho(A)$. Moreover $\lambda$ is a positive real number and has a positive real eigenvector \cite{gantmacher2005applications}. In this case, the eigenvector corresponding to $\lambda$ is unique up to scaling since $\lambda$ is a simple root of the characteristic polynomial of $A$ . In this article we only work with matrices with integer entries.

Given a non-negative integral matrix $A$, the \emph{adjacency graph} $\Gamma$ of $A$ is defined as follows. If $A = (a_{i,j})$ is $n \times n$, then $\Gamma$ has $n$ vertices $v_i$. Moreover, there are $a_{ij}$ oriented edges from $v_i$ to $v_j$. When $A$ is Perron-Frobenius, $\Gamma$ is path-connected by oriented paths. The following is well known.
\begin{lem}
Let $A$ be a Perron-Frobenius matrix. Denote by $R_i$ the sum of the entries in the $i$-th row of $A$. Then we have
\[\min_{i} \{ R_i \} \leq \rho(A) \leq \max_{i} \{ R_i \} \] 
\end{lem}
\noindent When $A$ is integral as well, the $(i,j)$ entry of the matrix $A^k$ is equal to the number of oriented paths of length $k$ from $v_i$ to $v_j$ in the adjacency graph of $A$. Therefore, using the above lemma for the matrix $A^k$, we deduce that 
\[ \rho(A)^k = \rho(A^k) \leq \max_{i} \hspace{1mm} \{ \Text{the number of oriented paths of length } k \Text{ starting from } v_i \}. \]

\begin{notation} 
Let $\Gamma$ be a directed graph with vertex set $V(\Gamma)$. For $v \in V(\Gamma)$, define $v^+$ to be the set of vertices $u$ such that there is an oriented edge from $v$ to $u$.
\end{notation}
The following technical lemma will be used in the proof of Theorem \ref{thm1} for bounding the spectral radius of certain matrices coming from pseudo-Anosov maps.

\begin{lem}
Let $A$ be a non-negative integral matrix, $\Gamma$ be the adjacency graph of $A$, and $V(\Gamma)$ be the set of vertices of $\Gamma$. Let $D$ and $k$ be fixed natural numbers. Assume the following conditions hold for $\Gamma$: 
\begin{enumerate}
\item  For each $v \in V(\Gamma)$ we have $\deg_{\Text{out}}(v) \leq D$, where $\deg_{\Text{out}}(v)$ denotes the outgoing degree of $v$. 

\item There is a partition $V(\Gamma) = V_1 \cup ... \cup V_k$ such that for each $v \in V_i$ we have $v^+ \subset V_{i+1}$, for any $1 \leq i \leq k$ except possibly when $i =1$ or $3$ (indices are$\mod k$). 

\item For each $v \in V_1$ we have $v^+ \subset V_2 \cup V_3$.

\item For each $v \in V_3$ we have $v^+ \subset V_3 \cup V_4$, and for $u \in v^+ \cap V_3$ we have $u^+ \subset V_4$.

\item For all $3 < j \leq k$ and each $v \in V_j$, the set $v^+$ consists of a single element.\\
\end{enumerate}
Then the spectral radius of $A^{k-1}$ is at most $\hspace{1mm} 4  D^4$.
\label{estimate}
\end{lem}
\begin{proof}
It is enough to show that for any vertex $v$, the number of directed paths of length $k-1$ and starting from $v$ is at most $4D^4$. Fix the vertex $v \in V$. We prove the lemma under the assumption that $v \in V_1$ as the other cases are similar. Suppose that starting at $v$, we want to construct a path of length $k-1$ by adding edges one by one. At each step we move from some $V_i$ to $V_{i+1}$ except when we are in $V_1$ or $V_3$. If we are in $V_1$ then we either move to $V_2$ or $V_3$. If we are in $V_3$ then we either stay in $V_3$ or move to $V_4$; however if we stay in $V_3$ then at the next step we definitely move to $V_4$. 

As $v \in V_1$, the path has one of the shapes in Figure \ref{Dynkin}. As an example, the diagram on the top means that the path starts at $V_1$ then it moves to $V_2$, then to $V_3$, it stays in $V_3$ at the next step, then moves to $V_4$ and then it moves to the next $V_i$ at each step until it gets to $V_{k-1}$, which is the ending point. 

Fix one of the diagrams in Figure \ref{Dynkin} for the shape of the path say the one on the top, and recall that the path is starting at $v$, and count the number of possible paths with the given property. At each step,
we have a unique choice of continuing the path one step further -- except when we are in one of $V_1 , V_2$ or $ V_3$. While being in $V_1 , V_2$ or $ V_3$, we have at most $D$ choices. Therefore, the number of possible paths is at most $1 \times D \times D \times D \times D\times1 \times \dots \times 1 = D^4$. Similarly the number of paths in the other three diagrams is at most $D^4$. Hence, the total number of paths is at most $4D^4$. This finishes the case that $v$ lies in $V_1$. 

Similarly, we can draw diagrams for the case when $v$ lies in each $V_j$ for $1<j \leq k$. Again there are at most four diagrams for such a path and the number of paths corresponding to each diagram is at most $D^4$. Therefore, the upper bound $4D^4$ for the number of paths holds in general no matter where the starting point lies.

\begin{figure}

\labellist
\pinlabel $1$ at 5 112
\pinlabel $2$ at 35 112
\pinlabel $3$ at 65 112
\pinlabel $3$ at 95 112
\pinlabel $4$ at 125 112
\pinlabel $\dots$ at 147 112
\pinlabel $k-1$ at 177 112

\pinlabel $1$ at 5 80 
\pinlabel $2$ at 35 80
\pinlabel $3$ at 65 80
\pinlabel $4$ at 95 80
\pinlabel $\dots$ at 120 80
\pinlabel $k$ at 147 80

\pinlabel $1$ at 5 48 
\pinlabel $3$ at 35 48
\pinlabel $3$ at 65 48
\pinlabel $4$ at 95 48
\pinlabel $\dots$ at 120 48
\pinlabel $k$ at 147 48

\pinlabel $1$ at 5 16 
\pinlabel $3$ at 35 16
\pinlabel $4$ at 65 16
\pinlabel $\dots$ at 90 16
\pinlabel $k$ at 112 16
\pinlabel $1$ at 144 16
\endlabellist

\centering
\includegraphics[width= 1.5 in]{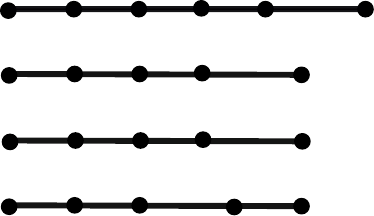}
\caption{}
\label{Dynkin}

\end{figure}

\end{proof}

\subsection{Thurston norm and fibered faces}\ \\
Thurston defined a natural semi-norm on the second relative homology of compact orientable 3-manifolds with real coefficients, now called the Thurston norm. We briefly discuss it here; the reader can see \cite{thurston1986norm} for a comprehensive treatment. Define the complexity of a compact, connected and oriented surface $S$ as 
\[\chi_-(S)=\max\{ -\chi(S),0 \}. \] 
If $S$ has multiple components, its complexity is defined as sum of the complexities of its components. Let $M$ be a compact, oriented 3-manifold. For $a \in H_2(M, \partial M; \mathbb{Z})$ the \emph{Thurston norm} of $a$, $x(a)$, is defined as:
\[ x(a):= \min \{ \chi_-(S) \hspace{1mm}|\hspace{1mm} [S]=a, \hspace{2mm} S \hspace{1mm} \Text{is compact, properly embedded, and oriented} \}. \]
The norm can be extended to rational points by linearity, and to real points by continuity. It was shown by Thurston that the unit ball of this norm is a convex polyhedron. Hence it makes sense to talk about faces of the Thurston norm.

From now on we work with closed 3-manifolds, and hence the Thurston norm is defined on $H_2(M ; \mathbb{R})$. We are interested in different ways that a single fibered 3-manifold $M$ can fiber over the circle. If $b_1(M) \geq 2$ and $M$ fibers over the circle then $M$ fibers in infinitely many different ways. Moreover, Thurston gave a structural picture of different ways that $M$ fibers in terms of its Thurston norm.
\begin{thm}[Thurston] Let $\mathcal{F}$ be the set of homology classes $\in H_2(M)$ that are representable by fibers of fibrations of $M$ over the circle.
\begin{enumerate}
\item Elements of $\mathcal{F}$ are in one-to-one correspondence with (non-zero) lattice points inside some union of cones on open faces of Thurston norm. 

\item If a surface $F$ is transverse to the suspension flow associated to some fibration of $M \longrightarrow S^1$ then $[F]$ lies inside the closure of the corresponding cone in $H_2(M)$.
\end{enumerate}
\label{thurston}
\end{thm}
\noindent The faces mentioned in Part (1) of the theorem are called $\textit{fibered faces}$.

\subsection{Bounds on entropy in a fibered cone} \hfill\\
Let $M$ be a closed, fibered 3-manifold with $b_1(M) \geq 2$ and $\mathcal{C}$ be a fibered cone of $H_2(M;\mathbb{R})$. The following theorem of Fried and Matsumoto extends the entropy function from lattice points to the entire fibered cone, and discusses the properties of the extension \cite{fried1982flow,fried1983transitive, matsumoto1987topological}. 
\begin{thm}[Fried-Matsumoto]
There is a strictly convex function $h \colon \mathcal{C} \longrightarrow \mathbb{R}$ such that:
\begin{enumerate}
\item for all $t >0$ and $u \in \mathcal{C}$ we have $h(tu) = (\frac{1}{t}) h(u)$.
\item for every primitive integral class $u \in \mathcal{C} \cap H_2(M)$, $h(u)$ is equal to the entropy of the monodromy corresponding to $u$.
\item $h(u) \longrightarrow \infty$ when $u \longrightarrow \partial \mathcal{C}$.
\end{enumerate}
\label{Fried}
\end{thm}
 \noindent Fried proved that the function is convex and Matsumoto showed that it is strictly convex. We will need the following \cite[Lemma 3.11]{agol2016pseudo}.
\begin{prop}[Agol-Leininger-Margalit]
Let $\mathcal{C}$ be a fibered cone for a mapping torus $M$, and $\overline{\mathcal{C}}$ be its closure in $H_2(M ; \mathbb{Z})$. If $u \in \mathcal{C}$ and $v \in \overline{\mathcal{C}}$, then $h(u+v) < h(u)$. 
\label{agol}
\end{prop}

\subsection{Jacobsthal function} \hfill\\
Let $n$ be a natural number. Every sequence of $n$ consecutive integers contains an element $a$ with $\gcd (a,n)=1$, where $\gcd$ stands for the greatest common divisor. Jacobsthal asked the following question \cite{jacobsthal1961sequenzen}: given $n$, what is the smallest integer $j(n)$ such that any sequence of $j(n)$ consecutive integers contains an element relatively prime to $n$? Therefore $j(n) \leq n$ by the later comment. Jacobsthal conjectured that 
\[ j(n) = O\Big ((\frac{\log(n)}{\log(\log(n))})^2 \Big ). \]
The best known upper bound is due to Iwaniec  \cite{iwaniec1978problem} who proved that 
\[ j(n) = O \big ( \log^2(n)\big ). \]

\begin{cor}
There exists a number $K \geq 2$ such that for every $n \geq 2$, each of the intervals 
\[ [\log^2(n), K \log^2(n)) , \hspace{2mm} [K \log^2(n), 2K \log^2(n)) , \hspace{2mm} [2K \log^2(n), 3K \log^2(n)), \hspace{2mm} \dots  \]
contains a number that is relatively prime to $n$.
\label{coprime}
\end{cor}

\subsection{Penner's construction of pseudo-Anosov maps}\hfil\\ 
Let $S$ be an orientable surface of genus $g$ with $n$ punctures. A \textit{multi curve} on $S$ is a union of distinct (up to isotopy) and disjoint simple closed curves on $S$. Let $\alpha = \alpha_1 \cup \dots \cup \alpha_m$ and $\beta = \beta_1 \cup \dots \cup \beta_n$ be two multi curves on $S$ such that $\alpha \cup \beta$ fills the surface, meaning that every component of $S -\alpha \cup \beta$ is either a disk or a once punctured disk (when $S$ is a punctured surface). Let $\tau_{\alpha_i}$ be the positive (right handed) \emph{Dehn twist} along $\alpha_i$. Define $\tau_{\beta_j}$ similarly.
 
Penner's theorem states that any word in $\tau_{\alpha_i}$ and $\tau_{\beta_j}^{-1}$ is pseudo-Anosov provided that all $\tau_{\alpha_i}$ and $\tau_{\beta_j}$ are used at least once \cite{penner1988construction}. Note that we are doing positive Dehn twist along the curves in one collection (multi curve) and negative twists along the other one. Furthermore, an \emph{invariant train track} for such a map can be obtained in the following way (independent of the chosen word): for each intersection of $\alpha$ with $\beta$, smooth the intersection as in Figure \ref{smoothcurves}. 

\begin{figure}
\labellist
\pinlabel $\alpha$ at 75 45 
\pinlabel $\beta$ at 48 70
\endlabellist
\centering
\includegraphics[width= 2 in]{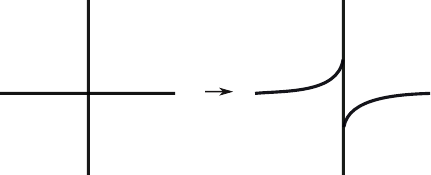}
\caption{Smoothing the intersections between multi curves}
\label{smoothcurves}
\end{figure}

\begin{ex}
Let $S$ be a closed, orientable surface of genus two, and the multi curves $\alpha = \alpha_1 \cup \alpha_2 \cup \alpha_3$ and $\beta = \beta_1 \cup \beta_2$ be as in Figure \ref{pennercurvesexample}. It can be easily seen that $\alpha \cup \beta$ fills the surface. By Penner's theorem, the map $f = \tau_{\alpha_1}^2 \circ \tau_{\alpha_2} \circ \tau_{\beta_2}^{-3} \circ \tau_{\alpha_3} \circ \tau_{\beta_1}^{-1} \circ \tau_{\alpha_1}$ is pseudo-Anosov. Moreover, an invariant train track for $f$ is shown in Figure \ref{pennercurvesexample}.
\end{ex}

\begin{figure}
\labellist
\pinlabel $\alpha_1$ at 17 115 
\pinlabel $\alpha_2$ at  100 115
\pinlabel $\alpha_3$ at 177 115
\pinlabel $\beta_1$ at 45 130
\pinlabel $\beta_2$ at 150 130
\endlabellist
\centering
\includegraphics[width= 2 in]{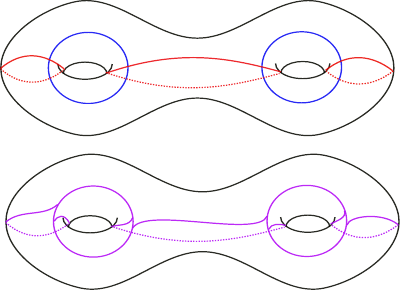}
\caption{Top: two multi curves that together fill the surface, Bottom: an invariant train track}
\label{pennercurvesexample}
\end{figure}
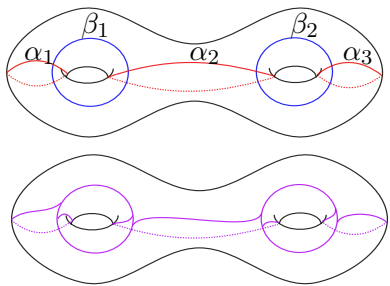

\label{penner-construction}

\subsection{Penner's examples of small dilatation pseudo-Anosov maps}In this section, we explain Penner's construction of pseudo-Anosov maps with small stretch factors. Imagine the surface $S_{g,0}$ as a sphere with $g$ symmetrically placed handles, and let $\rho$ be its natural $g$-fold symmetry (Figure \ref{penner-example}, right). Let $a,b$ and $c$ be the curves shown in Figure \ref{penner-example}, and denote the right handed Dehn twists along them by $\tau_a, \tau_b$ and $\tau_c$. Consider the map
\[ \phi : = \rho \circ \tau_c \circ \tau_a ^{-1} \circ \tau_b.  \]
It follows from Section \ref{penner-construction} that the map $\phi^g$ is pseudo-Anosov and an invariant train track $\tau$ for $\phi^g$ can be constructed by Penner's recipe. Let $V$ be the vector space of transverse measures on $\tau$. The map $\phi^g$ acts on $V$ and the spectral radius for this action is equal to the stretch factor of $\phi^g$. We can also explicitly compute the action of $\phi^g$ on $V$ (technically on an invariant subspace of $V$) and estimate, from above, the spectral radius of the action. In fact this upper bound can be chosen to be a constant independent of $g$, and Penner gave the upper bound of $11$. One the other hand, if we denote the stretch factor of $\phi$ by $\lambda(\phi)$ then
\[ \lambda(\phi^g)=\lambda(\phi)^g. \]
Hence we have
\[ \lambda(\phi)^g \leq 11 \implies \log(\lambda) \leq \frac{\log(11)}{g}, \]
which gives the upper bound. 

\begin{figure}
\labellist
\pinlabel $p$ at 53 77 
\pinlabel $q$ at 60 55
\pinlabel $a$ at 150 124
\pinlabel $b$ at 136 139
\pinlabel $c$ at 164 72
\pinlabel $\rho$ at 200 107

\endlabellist
\centering
\includegraphics[width= 2.5 in]{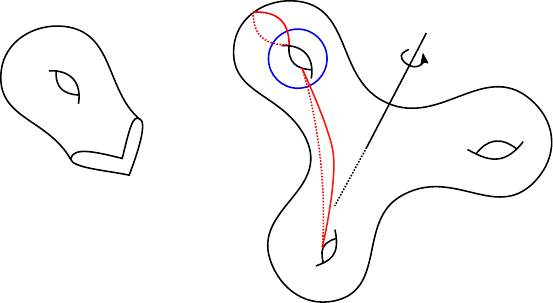}
\caption{Penner's examples of small dilatation pseudo-Anosov maps}
\label{penner-example}
\end{figure}
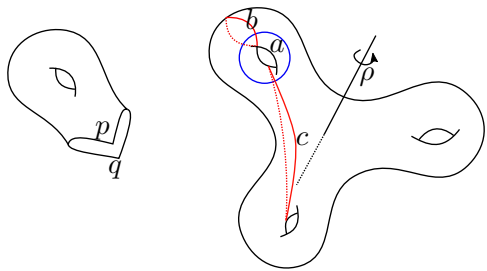

We can equivalently think about Penner's example as follows. Let $F$ be a surface of genus one with one boundary component and two marked points $p$ and $q$ on its boundary. One should think about $F$ as a fundamental domain for the rotational action of $\rho$ on $S_{g,0}$ (Figure \ref{penner-example}, left). The marked points can be thought of as the intersections of $S_{g,0}$ with the axis of the rotation $\rho$. Let $\alpha$ and $\beta$ be the oriented arcs on $\partial F$ going from $p$ to $q$. For each $i \in \mathbb{Z}$, let $F_i$ be a copy of $F$. Consider the infinite surface $S_\infty$ obtained by taking the union of $F_i$ and gluing them together
\[ S_\infty := (\cup F_i)/\sim, \]
where the relation $\sim$ identifies $\alpha_i$ with $\beta_{i+1}$ for each $i \in \mathbb{Z}$. 

There is a shift action $\tilde{\rho}$ on $S_\infty$ that sends each $F_i$ to $F_{i+1}$. It is clear that the quotient of the surface $S_\infty$ by the subgroup $\langle(\tilde{\rho})^g \rangle$ is homeomorphic to $S_{g,0}$. Moreover, if we define the curves $\tilde{a}$, $ \tilde{b}$ and $\tilde{c}$ inside $\big( (F_0 \cup F_1)/\sim \big) \hspace{1mm} \subset S_\infty$, then they project to the curves $a$, $b$ and $c$ under the projection $S_\infty \longrightarrow S_{g,0}$. 

Had we assumed in the beginning that one (respectively both) of the points $p$ and $q$ is a puncture, then we would have obtained a similar construction with the projection $S_\infty \longrightarrow S_{g,1}$ (respectively $S_\infty \longrightarrow S_{g,2}$).
This perspective is useful for us, since we are going to adjust Penner's construction by considering actions of the group $\mathbb{Z} \times \mathbb{Z}$ on a suitable infinite surface (generated by two independent and commuting shifts).  
\label{Penner-small-dilatation}
\section{Construction of the maps}

\begin{thm}
There is a universal constant $C>0$ such that for all $n \geq 1$ and $g \geq 17n+1$, we have the following inequality
\[ l_{g,n} \leq C \hspace{1mm} \frac{n}{g}. \]
\label{thm1}
\end{thm}

\begin{proof} \ \\
\textbf{Outline:}
First, we construct `small dilatation' maps using Penner's construction, on a sequence of surfaces $P_{n,k}$ where $n \geq 1$ and $k \geq 3$. Denote these maps as 
\[ f_{n,k}: P_{n,k} \longrightarrow P_{n,k}. \] 
The surface $P_{n,k}$ has exactly $n$ punctures. Let $g_{n,k}$ be the genus of the surface $P_{n,k}$. Therefore the set 
\[ \mathcal{W}_n :=\{ g_{n,k} \hspace{1mm}|\hspace{1mm}  k \geq 3 \}, \] 
indicates the set of genera obtained from this construction for exactly $n$ punctures. If the set $\mathcal{W}_n$ had contained all natural numbers larger than $17n$, we would have been done; however this is not the case. We need to fill the remaining natural numbers that are missing from the set $\mathcal{W}_n$ by a different construction. Let $M_{n,k}$ be the mapping torus of $f_{n,k}$. Denote by 
\[ \mathcal{C}_{n,k} \subset H_2(M_{n,k}; \mathbb{R}), \]
the Thurston fibered face corresponding to the monodromy map $f_{n,k}$. The following property is crucial for us: 

There exists a closed orientable surface $F_{n,k}$ of genus two in the closure of the fibered face $\mathcal{C}_{n,k}$. 

We look at the fibered face $\mathcal{C}_{n,k}$ and use the surface $F_{n,k}$ to construct new fibrations of the manifold $M_{n,k}$. The new fibrations give us the missing numbers between $g_{n,k}$ and $g_{n, k+1}$; here the hypothesis that $F_{n,k}$ has genus two will be used.

Another important property is that the ratio $\frac{g_{n,k+1}}{g_{n,k}}$ is bounded above by a universal constant independent of $g$ and $k$. This will be used in proving that the new constructed maps that give the missing genera, have `small' stretch factors as well.\\ 

\textbf{Step 1}: Define the surface $P_{n,k}$ in the following way. Let $T$ be an orientable surface of genus $5$ with $3$ boundary components $c$, $ d$ and $e$. We give $c$, $d$ and $e$ the induced orientations from $T$. Let $p$ (respectively $q$) be a puncture (respectively a marked point) on the boundary component $e$ of $T$. Let $r$ and $s$ be the two oriented arcs connecting $p$ and $q$ in $\partial T$, with the induced orientations from $e$. See Figure \ref{fundamental-region} for a picture of $T$. Let $T_{i,j}$ be copies of the surface $T$, where $i,j \in \mathbb{Z}$. We use similar notations to refer to the boundary components of $T_{i,j}$. Define the infinite surface $S_{\infty}$ as the quotient
\[ S_\infty := \big( \bigcup T_{i,j} \big)/\sim, \]
where $i,j \in \mathbb{Z}$. The equivalence relation $\sim$ is defined as
\[ c_{i,j} \sim d_{i+1,j} \hspace{3mm}, \hspace{3mm} r_{i,j} \sim s_{i,j+1}, \]
where $i,j \in\mathbb{Z}$, and the gluing maps for $c_{i,j} \sim d_{i+1,j}$ and $r_{i,j} \sim s_{i,j+1}$ are by orientation-reversing homeomorphisms.

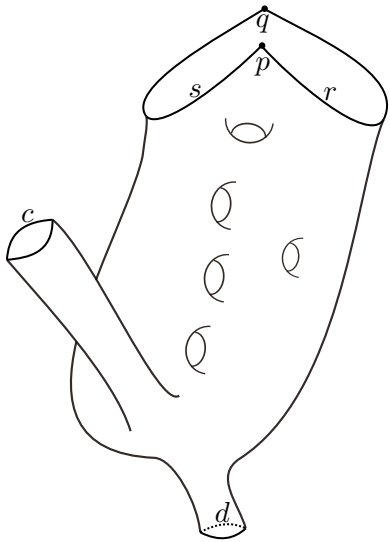
\begin{figure}
\labellist
\pinlabel $p$ at 95 175
\pinlabel $q$ at 95 191
\pinlabel $r$ at 120 165
\pinlabel $s$ at 70 166
\pinlabel $c$ at 8 120
\pinlabel $d$ at 80 10 
\endlabellist

\includegraphics[width= 2 in]{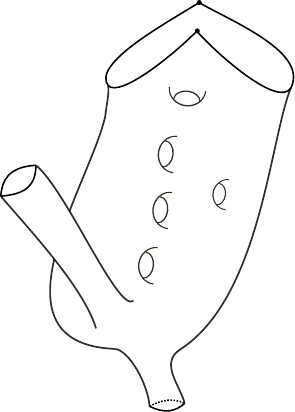}
\caption{The surface $T$, which is the building block.}
\label{fundamental-region}
\end{figure}

There are two natural maps $\overline{\rho_1}, \overline{\rho_2} : S_\infty \longrightarrow S_\infty$ that act by shifts as follows 
\[ \overline{\rho_1} \text{ sends } T_{i,j} \text{ to }T_{i+1, j}, \]
\[ \overline{\rho_2} \text{ sends } T_{i,j} \text{ to } T_{i,j+1}. \]
Note that the maps $\overline{\rho_1}$ and $\overline{\rho_2}$ commute. Define the surface $P_{n,k}$ as the quotient of the surface $S_\infty$ by the covering action of the group generated by $(\overline{\rho_1})^n$ and $ (\overline{\rho_2})^k$. Therefore, $\overline{\rho_1}$ and $\overline{\rho_2}$ induce maps on the surface $P_{n,k}$, which we denote by $\rho_1$ and $\rho_2$. Here we have a slight abuse of notation by suppressing the indices $n$ and $k$ for the maps $\rho_1$ and $\rho_2$.
\begin{lem}
Define the sequence 
\begin{equation}
g_{n,k} = (6k-1)n+1, \hspace{5mm}k \geq 3, \hspace{2mm} n \geq 1.
\end{equation}
The genus of $P_{n,k}$ is equal to $g_{n,k}$. 
\end{lem}

\begin{proof}
Consider the subsurface $U \subset P_{n,k}$ defined as
\[ U = \big( \bigcup_{i=0}^{k-1} T_{0,i} \big) / \sim. \]
Then $U$ is a compact, orientable surface of genus $5k$ with $2k$ boundary components, and forms a fundamental domain for the covering action of $\overline{\rho_1}$ on $S_\infty$. We have 
\[ \chi(U) = 2- 2(5k)-2k=2-12k. \]
Moreover 
\[ \chi(P_{n,k}) = n \cdot \chi(U) = -2n(6k-1), \]
since $P_{n,k}$ is formed by gluing $n$ copies of $U$ together along circle boundary components which have zero Euler characteristic. Therefore 
\[ \chi(P_{n,k}) = 2-2g_{n,k}=  -2n(6k-1) \implies g_{n,k} = n(6k-1)+1.\] 
\end{proof}

\begin{figure}
\labellist
\pinlabel $\beta_1$ at 100 94
\pinlabel $\beta_5$ at 114 100
\pinlabel $\alpha_4$ at 87 94
\pinlabel $\alpha_1$ at 59 87
\pinlabel $\beta_6$ at 54 102
\pinlabel $\gamma$ at 68 122
\pinlabel $\beta_2$ at 106 164
\pinlabel $\alpha_2$ at 81 160 
\pinlabel $\alpha_3$ at 121 135
\pinlabel $\beta_4$ at 109 128
\pinlabel $\beta_3$ at 36 72
\pinlabel $\alpha_5$ at 20 95
\pinlabel $\alpha_6$ at 106 79
\endlabellist
\centering
\includegraphics[width= 3 in]{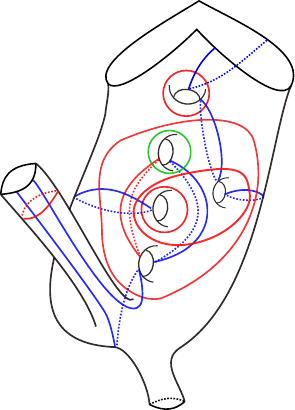}
\caption{The portion of $\alpha$ and $\beta$ curves that lie on $T_{0,0}$.}
\label{curves}
\end{figure}

\begin{figure}
\labellist
\pinlabel $\beta_3$ at 72 35
\endlabellist
\centering
\includegraphics[width= 2 in]{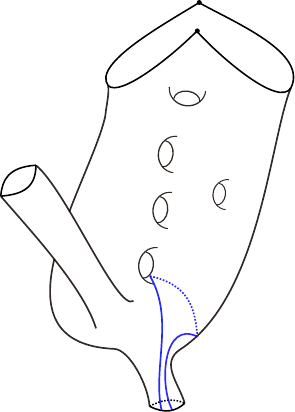}
\caption{The rest of the curve $\beta_3$ on $T_{1,0}$}
\label{curves2}
\end{figure}

\begin{figure}
\labellist
\pinlabel $\beta_2$ at 72 155
\endlabellist
\centering
\includegraphics[width= 2 in]{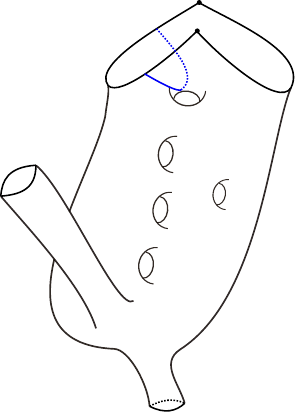}
\caption{The rest of the curve $\beta_2$ on $T_{0,1}$}
\label{curves3}
\end{figure}

\textbf{Step 2}: The map $f_{n,k}: P_{n,k} \longrightarrow P_{n,k}$ will be defined as a composition of suitable Dehn twists, followed by a finite order mapping class. Let us specify the curves along which we do the Dehn twists. 

Let $\mathcal{B}$ be the union of all $\beta$ curves except $\beta_1$ in $T_{0,0} \cup T_{0,1} \cup T_{1,0}$ (see Figures \ref{curves}, \ref{curves2} and \ref{curves3}). Let $\rho_1(\mathcal{B})$ be the image of $\mathcal{B}$ under $\rho_1$; we use the notation $\rho_1^\ell(\mathcal{B})$ for $\ell \in \mathbb{N}$ similarly. Define $\phi _{b}$ as the composition of positive Dehn twists along all the curves in the set $\overline{\mathcal{B}} : = \mathcal{B} \cup \rho_1(\mathcal{B}) \cup \dots \cup \rho_1^{n-1}(\mathcal{B})$. Since the curves in $\overline{\mathcal{B}}$ are disjoint, Dehn twists along them commute. Therefore, it is not necessary to specify the order in which we compose these Dehn twists in $\phi_b$. Let $\mathcal{R}$ be the union of all $\alpha$ curves except $\alpha_1$ in $T_{0,0}$. Define $\overline{\mathcal{R}}$ and $\phi_r$ similarly but use negative Dehn twists this time. 

Let $\alpha_1 , \beta_1 \subset T_{0,0}$ be the curves in Figure \ref{curves}. Let $\phi$ be the composition of negative Dehn twists along all the curves $\alpha_1 , \rho_1(\alpha_1) \cdots, \rho_1^{n-1}(\alpha_1)$ followed by positive Dehn twists along all the curves $\beta_1 , \rho_1(\beta_1) , \cdots, \rho_1^{n-1}(\beta_1)$. Define
\begin{equation}
f_{n,k} := \rho_2 \circ \phi \circ \phi _{b} \circ \phi_{r}. 
\end{equation}

Here in our notation the composition is from right to left. It follows from Penner's construction of pseudo-Anosov maps that $(f_{n,k})^k$ is pseudo-Anosov. Hence $f_{n,k}$ itself is pseudo-Anosov and an invariant train track $\tau_{n,k}$ for $f_{n,k}$ can be obtained from Penner's construction by smoothing the intersection points. \\

\textbf{Step 3}: Let $M_{n,k}$ be the mapping torus of $f_{n,k}$ and $\mathcal{C}_{n,k}$ be the fibered face of $H_2(M_{n,k} , \mathbb{R})$ corresponding to the map $f_{n,k}$. We show that the closure of $\mathcal{C}_{n,k}$ contains a closed orientable surface of genus two. 

\begin{lem}
There is a non-trivial homology class $0 \neq [F_{n,k}] \in H_2(M_{n,k}; \mathbb{Z})$ that is represented by an orientable surface $F_{n,k}$ of genus two. Moreover, $F_{n,k}$ is Thurston norm-minimizing and lies in the closure $ \overline{\mathcal{C}_{n,k}}$.
\label{genustwosurface}
\end{lem}

\begin{proof}
To simplify the notation, we drop the subscripts $n$, $k$ from $M$, $ F$ and $\mathcal{C}$ in the proof. Let $\gamma \subset T_{0,0}$ be the curve as shown in Figure \ref{curves}. We use $\gamma$ to construct the surface $F$. Let us follow the image of $\gamma$ under the iterations of $f$. Let $\hat{\gamma} \subset T_{0,0}$ be the curve shown in Figure \ref{genustwo}. It is easy to check that $ \hat{\gamma} = \phi(\gamma)$. Hence we have
\begin{align*}
f(\gamma) &=  \rho_2 \circ \phi \circ \phi _{b} \circ \phi_{r} (\gamma) =\rho_2 \circ \phi(\gamma) = \rho_2(\hat{\gamma})\\
f^2(\gamma) &= \rho_2^2(\hat{\gamma}) \\
& \vdots \\
f^k(\gamma) &= \rho_2^k(\hat{\gamma}) = \hat{\gamma}
\end{align*}  
Note that $\gamma - \hat{\gamma}$ bounds an orientable surface $\hat{F}$ of genus one (see Figure \ref{genustwo}). Therefore one can assemble $\gamma$ , $f(\gamma)$, \dots, $f^k(\gamma)$ together and obtain the desired surface as follows: Let $T_i$ be a tube that connects $f^{i-1}(\gamma)$ to $f^{i}(\gamma)$ in the mapping torus $M$, for $1 \leq i \leq k$. The tube $T_i$ is obtained by following the curve $f^{i-1}(\gamma)$ along the suspension flow of $f$. Define $F$ as the union of $T_1 , T_2 , \dots , T_k$ and $\hat{F}$. Since $T_i$ are tubes and $\hat{F}$ has genus one, the resulting surface is a closed surface of genus two. It can be easily seen that $F$ is embedded and orientable. 

We show that the surface $F$ can be isotoped to be transverse to the suspension flow of $f$, and therefore $[F] \in \overline{\mathcal{C}}$ by Theorem \ref{thurston}. The proof is as in \cite[Lemma 5.1]{leininger2012number}. Let $N(\gamma)$ be a tubular neighbourhood of $\gamma$ in $\hat{F}$, and $\eta \colon \hat{F} \rightarrow [0,1]$ be a smooth function supported on $N(\gamma)$ with $\eta^{-1}(1)=\gamma$ and such that the derivative of $\eta$ on $\gamma$ vanishes. Denote the suspension flow of the map $f$ by $\Phi_t \colon M \rightarrow M$, where $t \in \mathbb{R}$, and define the map $g \colon \hat{F} \rightarrow M$ as $g(x) = \Phi_{k \cdot\eta(x)}(x)$. Then $g$ restricted to the interior of $\hat{F}$ is an embedding, and satisfies $g(\gamma) = \hat{\gamma}$. The image of $g \colon \hat{F} \rightarrow M$ is an embedded surface of genus two, that is isotopic to the natural embedding of $F$ in $M$, and is transverse to the suspension flow.

The surface $F$ is norm-minimizing, otherwise if $K$ is a surface representing the homology class $[F]$  and $\chi_-(K) < \chi_-(F)$ then $K$ should contain an essential torus or sphere. But there can not be any essential torus or sphere in $M$ since $f$ is pseudo-Anosov. More precisely, the universal cover of $M$ is homeomorphic to $\mathbb{R}^3$ and is irreducible by Alexander's theorem \cite[Theorem 9.2.10]{martelli2016introduction}, hence so is $M$ \cite[Proposition 9.2.16]{martelli2016introduction}. See \cite[Proposition 14.9]{fathi1979travaux} for a proof of $M$ being atoroidal.

\begin{figure}
\labellist
\pinlabel $\gamma$ at 58 105
\pinlabel $\hat{\gamma}$ at 55 80
\endlabellist
\centering
\includegraphics[width= 1.5 in]{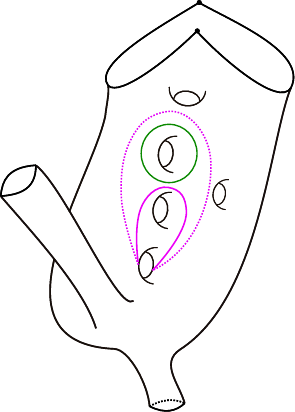}
\caption{The curves $\gamma, \hat{\gamma} \subset T_{0,0}$ together bound an orientable surface of genus one.}
\label{genustwo}
\end{figure}

\end{proof}

\textbf{Step 4:}
\begin{lem}
Let $\lambda_{n,k}$ be the stretch factor of the map $f_{n,k}$. There is a universal constant $C'>0$ such that for every $n \geq 1$ and $k \geq 3$ we have
\[ \lambda_{n,k} \leq  C' \frac{n}{g_{n,k}}. \]
\label{stretch}
\end{lem}

\begin{proof}
Let $\tau_{n,k}$ be the invariant train track for $f_{n,k}$ obtained from Penner's construction. We drop the subscripts from $\tau_{n,k}$ for convenience. Define the multi curves
\[ \mathcal{A} := \mathcal{B} \cup \mathcal{R} \cup \{ \alpha_1 , \beta_1 \}, \hspace{3mm} \overline{\mathcal{A}} := \mathcal{A} \cup \rho_1(\mathcal{A}) \cup \dots \cup \rho_1^{n-1}(\mathcal{A}) \]
\[ \hat{\mathcal{A}} := \overline{\mathcal{A}} \cup \rho_2(\overline{\mathcal{A}}) \cup \dots \cup \rho_2^{k-1}(\overline{\mathcal{A}}).\] 
For each connected curve $x \subset \hat{\mathcal{A}}$, there is an associated transverse measure $\mu_{x}$ for $\tau$. By definition, $\mu_x$ assigns $1$ to all edges that lie in $x$ and assigns $0$ to every other edge. Let $V_{\tau}$ be the cone of transverse measures on $\tau$, and $H$ be the subspace of $V_{\tau}$ spanned by the elements
\[ \{ \mu_x \hspace{1mm} | \hspace{1mm} x \Text{ is a connected curve in } \hat{\mathcal{A}} \}. \] 
It is easy to see that $\mu_x$ are indeed linearly independent, and we refer to this basis for $H$ as the \emph{standard basis}. The subspace $H$ is invariant under the action of $f_{n,k}$ on $V_\tau$, and $\lambda_{n,k}$ is equal to the Perron-Frobenius eigenvalue of the action of $f_{n,k}$ on $H$. 

Let $A$ be the matrix representing the linear action of $f_{n,k}$ on $H$ in the standard basis. Denote by $\Gamma$ the adjacency graph corresponding to $A$. We want to show that $A$ has `small' spectral radius. The intuition is that $\Gamma$ is `sparse' since `most' connected curves in $\hat{\mathcal{A}}$ just get rotated by the action of $f_{n,k}$, and $\Gamma$ has no cycle of `small' length. More precisely, we have the following partition
\[ \hat{\mathcal{A}}= \bigcup_{i=1}^k \rho_2^{i-2}(\overline{\mathcal{A}}). \]
Define $V_i$ for $1 \leq i \leq k$ as those vertices of $\Gamma$ that correspond to the elements 
\[ \{ \mu_y \hspace{1mm} | \hspace{1mm} y \text{ is a connected curve in } \rho_2^{i-2}(\overline{\mathcal{A}}) \} .\]
We check the conditions of  Lemma \ref{estimate}, based on the combinatorics of the curves in $\hat{\mathcal{A}}$.
\begin{enumerate}
\item There exists a universal constant $D'$, independent of $k$ and $ n$, such that for every connected curve $c$ in $\hat{\mathcal{A}}$, the geometric intersection number between $c$ and $\bar{\mathcal{A}}$ is at most $D'$. We can write 
\[ A = M_4 M_3 M_2 M_1, \]
where $M_4, M_3, M_2$ and $M_1$ show the action of $\rho_2, \phi, \phi_b$ and $\phi_r$ on $H$ respectively. The $L^1$-norm of $A(\mu_x)$ for a connected curve $x \in \hat{\mathcal{A}}$ is at most $D = (1+D')^3$. This is because each of the matrices $M_1, M_2$ and $M_3$ change the $L^1$-norm by a factor of at most $1+D'$, and $M_4$ preserves the $L^1$-norm. This shows that the outward degree of each vertex in the adjacency graph is at most $D$. 

\item  Let $v \in V_i$ be a vertex corresponding to $\mu_c$ for a curve $c \in \hat{\mathcal{A}}$, where $i \neq 1,3$. Note $f_{n,k}$ is defined as $f_{n,k} = \rho_2 \circ \phi \circ \phi _{b} \circ \phi_{r} $. Then the action of $\phi \circ \phi _{b} \circ \phi_{r}$  sends $\mu_c$ to a sum of $\mu_y$ where $y$ corresponds to elements of $V_i$. Moreover, $\rho_2$ sends $\mu_y$ to $\mu_z$ where $z$ corresponds to an element of $V_{i+1}$.

\item The only elements $v \in V_1$ such that $v^+ \not \subset V_2$ are the ones corresponding to 
\[  X = \{ \mu_y \hspace{1mm} |\hspace{1mm}  \exists i \text{ s.t. } y \text{ is a connected curve in }  \rho_1^i(\rho_{2}^{-1}(\beta_2)) \}. \]
For any element $v \in V_1$ corresponding to $X$, we have $v^+ \subset V_2 \cup V_3$.  

\item The only elements $v \in V_3$ such that $v^+ \not \subset V_4$ are the ones corresponding to
\[  Y = \{ \mu_y \hspace{1mm} |\hspace{1mm}  \exists i \text{ s.t. } y \text{ is a connected curve in }  \rho_1^i(\rho_2(\alpha_2)) \} .\]
Moreover, for any element $v \in V_3$ corresponding to $Y$ and any $u \in v^+ \cap V_3$, the element $u$ does not correspond to $Y$ any more, and hence $u^+ \subset V_4$.

\item Every curve corresponding to an element of $V_j$, $3<j \leq k$, is disjoint from all curves in $\overline{\mathcal{A}}$. Therefore it just gets rotated by $\rho_2$. 
\end{enumerate}
Setting $\lambda= \lambda_{n,k}$, Lemma \ref{estimate} implies that 
\[ \lambda^{k-1} = \rho(A)^{k-1} = \rho(A^{k-1}) \leq 4D^4 \implies (k-1) \cdot \log(\lambda) \leq  \log(4D^4) \]
\[ \implies \frac{k}{2} \log(\lambda) \leq (k-1)\log(\lambda) \leq \log(4D^4). \]
On the other hand $g_{n,k} = (6k-1)n+1 \leq 6kn$. Therefore
\[  \log(\lambda) \leq 2\log(4D^4) \cdot \frac{1}{k} \leq 2\log(4D^4) \cdot \frac{6n}{g_{n,k}} = C'  \hspace{1mm} \frac{n}{g_{n,k}},\]
where $C' := 12 \log(4D^4)$.
\end{proof}

\textbf{Step 5:} We want to use the mapping torus of $f_{n,k} : P_{n,k} \longrightarrow P_{n,k}$ to construct pseudo-Anosov maps on surfaces of genus $g_{n,k} \leq g \leq g_{n,k+1}$ with `small' stretch factors. These maps should still keep invariant $n$ of the singularities of their invariant foliations. For each integer $r \geq 0$, consider the homology class $[P_{n,k}^r] := [P_{n,k}] + r[F_{n,k}]$, where $F_{n,k}$ is the orientable surface of genus two constructed in Lemma \ref{genustwosurface}. A representative $P_{n,k}^r$ for the homology class $[P_{n,k}^r]$ can be obtained by taking the oriented sum (cut and paste) of the surface $P_{n,k}$ and $r$ copies of the surface $F_{n,k}$. 
\begin{lem}
The surface $P_{n,k}^r$ is Thurston norm-minimizing, and its genus is equal to $g_{n,k}^r: = g_{n,k}+r$. In particular as $r$ varies between $0$ and $6n$, the genera of $P_{n,k}^r$ cover the range between $g_{n,k}$ and $g_{n,k+1}$. Moreover, $P_{n,k}^r$ is the fiber of a fibration of $M_{n,k}$ with pseudo-Anosov monodromy that fixes $2n$ of the singularities of its invariant foliation.
\label{cover}
\end{lem}

\begin{proof}
We have 
\[ \chi(P_{n,k}^r) = \chi(P_{n,k})+ r \cdot \chi(F_{n,k}) = (-2 g_{n,k}+2) - 2r = -2(g_{n,k}+r)+2.   \]
This proves the identity for the genus of $P_{n,k}^r$. To see that $P_{n,k}^r$ is norm-minimizing, note that $[P_{n,k}^r] \subset \mathcal{C}_{n,k}$. By linearity of the Thurston norm on a fibered face, we have
\[
x ([P_{n,k}^r])  = x \big([P_{n,k}] + r[F_{n,k}] \big) = x([P_{n,k}] )+ r x([F_{n,k}]) = \chi_-(P_{n,k}) + 2r = 2(g_{n,k}+r)-2.
\]
The identity $g_{n,k} = (6k-1)n+1$ implies that 
\[ g_{n,k+1} - g_{n,k} = 6n. \]
Hence, as $r$ varies between $0$ and $6n$, the genera of $P_{n,k}^r$ cover the range between $g_{n,k}$ and $g_{n,k+1}$.

The homology class $[P_{n,k}^r]=[P_{n,k}] + r[F_{n,k}]$ is clearly integral. It is primitive as well, since there is a curve in $M_{n,k}$ that intersects $P_{n,k}$ transversely and exactly once, while avoiding $F_{n,k}$. The class $[P_{n,k}^r]$ is integral and primitive and lies in the fibered face $\mathcal{C}_{n,k}$, and hence by Theorem \ref{thurston} is the fiber of a fibration of $M_{n,k}$. The monodromy $f_{n,k}$ corresponding to one fibration for $\mathcal{C}_{n,k}$ is pseudo-Anosov, hence every monodromy corresponding to this face and coming from the first return map of the suspension flow is pseudo-Anosov \cite[Lemma 14.12]{fathi1979travaux}. In particular, $f_{n,k}^r$ is pseudo-Anosov.

Note that the singularities of the stable foliation of $f_{n,k}$ that are fixed by the map $f_{n,k}$, are the $2n$ intersection points of the axis of $\rho_1$ with $P_{n,k}$. Moreover the surface $F_{n,k}$ can be isotoped to be transverse to the suspension flow and be disjoint from the orbit of these $2n$ singularities (see the proof of Lemma \ref{genustwosurface}). This shows that if we look at the monodromy $f_{n,k}^r$ of the fibration of $P_{n,k}^r$, the corresponding $2n$ singularities are still fixed by $f_{n,k}^r$.
\end{proof}

\begin{lem}
There is a constant $C>0$ such that for every $n \geq 1$, $k \geq 3$, and $0 \leq r \leq 6n$ we have
\[ \lambda_{n,k}^r \leq C \frac{n}{g_{n,k}^r}. \]
\label{stretch-complete}
\end{lem}

\begin{proof}

Let $\mathcal{C} = \mathcal{C}_{n,k}$, and denote by $h: \mathcal{C} \longrightarrow \mathcal{R}$ the function coinciding with the entropy function for primitive integral points, given by Theorem \ref{Fried}. Note we have
\begin{equation}  
g_{n,k}^r = g_{n,k}+r \leq g_{n,k}+6n < 2 g_{n,k},
\label{comparison}
\end{equation}
since 
\[ g_{n,k} = (6k-1)n +1 >6n. \]
Hence 
\[ h([P_{n,k}^r]) < h([P_{n,k}]) \leq C' \hspace{1mm} \frac{n}{g_{n,k}}\leq 2C' \hspace{1mm} \frac{n}{g_{n,k}^r}, \]
where the first inequality is by Proposition \ref{agol}, the second inequality by Lemma \ref{stretch}, and the third inequality by (\ref{comparison}). Therefore, we can take $C = 2C'$.

\end{proof}

\textbf{Final Step:} Steps 1 and 2 define the surfaces $P_{n,k}$ (with $n$ punctures) and pseudo-Anosov mapping classes
\[ f_{n,k} \colon P_{n,k} \longrightarrow P_{n,k}, \]
for $n \geq 1$ and $k \geq 3$. Denote the genus of $P_{n,k}$ by $g_{n,k}$ and the stretch factor of $f_{n,k}$ by $\lambda_{n,k}$. By Lemma \ref{stretch}, there is a universal constant $C'>0$ such that 
\[ \lambda_{n,k} \leq C' \frac{n}{g_{n,k}}.\]
For any $0 \leq r \leq 6n$, define the surface $P_{n,k}^r$ with genus $g_{n,k}^r$ and the pseudo-Anosov mapping class
\[ f_{n,k}^r  \colon P_{n,k}^r \longrightarrow P_{n,k}^r, \] 
as in Step 5. By Lemma \ref{stretch-complete}, the map $f_{n,k}^r$ fixes $2n$ singularities of its invariant foliation. Hence by puncturing $P_{n,k}^r$ at $n$ of these singularities, we can think of $f_{n,k}^r$ as defined on a surface of genus $g_{n,k}^r$ with $n$ punctures. Denote the stretch factor of $f_{n,k}^r$ by $\lambda_{n,k}^r$. By Lemma \ref{stretch-complete}, there is a universal constant $C>0$ such that for any $n \geq 1$, $k \geq 3$, and $0 \leq r \leq 6n$ we have
\[ \lambda_{n,k}^r \leq C \frac{n}{g_{n,k}^r}. \]
Moreover by Lemma \ref{cover}, as $r$ varies between $0$ and $6n$, the genera $g_{n,k}^r$ cover all natural numbers between $g_{n,k}$ and $g_{n,k+1}$. Therefore for each fixed $n \geq 1$, the set 
\[ S_n = \{ g_{n,k}^r \hspace{1mm}| \hspace{1mm} k \geq 3, \hspace{2mm} 0 \leq r \leq 6n \}, \]
coincides with the set of natural numbers larger than or equal to $17n+1 = g_{n,3}$. Hence we have enough `small dilatation' examples to conclude the theorem with $C$ as in Lemma \ref{stretch-complete}.
\end{proof}

\begin{remark}
The idea of using different fibrations of the same 3-manifold to construct many small dilatation pseudo-Anosov maps was first used by McMullen in \cite{mcmullen2000polynomial}. The idea of constructing a closed, orientable surface of genus two in the closure of the fibered face is borrowed from \cite{leininger2012number, agol2016pseudo}. 
\end{remark}

\begin{remark}
A property of the constructed examples is 
\[ b_1(M_{n,k}) \geq n+1, \]
where $b_1$ denotes the first Betti number. In fact if we start with the curve $\rho_{1}^{i-1}(\gamma)$ instead of $\gamma$ in Lemma \ref{genustwosurface} for $1 \leq i \leq n$, we can construct an orientable surface, $F_{n,k}^i$, of genus two in $H_2(M_{n,k};\mathbb{Z})$. Then the homology classes of the surfaces 
\[ P_{n,k}, \hspace{1mm} F_{n,k}^1, \hspace{1mm} F_{n,k}^2 , \cdots, F_{n,k}^n  \]
are linearly independent. This is because for any member of the above list, one can find a curve in $M_{n,k}$ that is disjoint from all of them except the singled out member, and transversely intersects the remaining member exactly at one point. Compare with the lower bound for dilatation in \cite{agol2016pseudo} under the the assumption that $b_1$ is bounded from below.
\end{remark}

\newtheorem*{thm:main}{Theorem \ref{thm:main}}
\begin{thm:main}
For any fixed $n \in \mathbb{N}$, there are positive constants $B_1 = B_1(n)$ and $B_2=B_2(n)$ such that for any $g \geq 2$
\[ \frac{B_1}{g} \leq l_{g,n} \leq \frac{B_2}{g}. \] 
\end{thm:main}

\begin{proof}
First we prove the upper bound. By Theorem \ref{thm1} there exists a number $C>0$ such that for $g \geq 17n+1$ we have 
\[  l_{g,n} \leq C \hspace{1mm} \frac{n}{g}. \]
Hence we may take $B_2(n)$ equal to 
\[ \max \{ \hspace{1mm} Cn , \hspace{1mm}\hspace{1mm} l_{1,n} \hspace{1mm}, \hspace{1mm} 2 \hspace{1mm}l_{2,n} \hspace{1mm}, 3 \hspace{1mm}l_{3,n}, ... , \hspace{1mm} 17n \hspace{1mm} l_{17n,n} \hspace{1mm} \}. \]
\noindent For the lower bound, Penner \cite{penner1991bounds} has proved the following general result
\[ \l_{g,n} \geq \frac{\log(2)}{12g-12+4n}.  \]
Therefore, one can take $B_1(n) = \frac{\log(2)}{12n}$ since 
\[ \frac{\log(2)}{12g-12+4n} > \frac{\log(2)}{12ng}. \]
\end{proof}

\begin{cor}
Fix a natural number $n$. The asymptotic behavior of $l_{g,n}$ for $g$ varying is like $\frac{1}{g}$. 
\end{cor}

\noindent In \cite{penner1991bounds}, Penner proved that
\[ \frac{\log(2)}{12g-12} \leq l_{g,0} \leq \frac{\log(11)}{g} .\]
He mentions: ``The results of \cite{bauer1992upper} can be used to sharpen the upper bound when $n=0$, and we expect that they can also be used to give an analogous upper bound in case $n \neq 0$." The above corollary confirms that such an upper bound exists, however our method is different from \cite{bauer1992upper}.




\section{Uniform Bounds}

In this section, we improve the upper bound in Theorem \ref{thm:main} and give an upper bound independent of the number of punctures. The trade off is the assumption that the genus is `large' compared to the number of punctures.



\newtheorem*{thm:uniform}{Theorem \ref{thm:uniform}}
\begin{thm:uniform}
There exist universal positive constants $A$, $B$ and $C$ such that for any natural number $n $ and any natural number $g \geq C \hspace{1mm} n\log^2(n)$ we have
\[ \frac{B}{g} \leq l_{g,n} \leq \frac{A}{g}. \]
\end{thm:uniform}

\begin{proof}
Before starting the proof, we should give a warning that some of the notation used in the proof of this theorem is similar to those used in the proof of Theorem \ref{thm1}, including $f_{n,k}$ and $\lambda_{n,k}$. However, in general the definitions are different from before unless otherwise stated, and one should not confuse them.

Fix a natural number $n$. We can assume that $n \geq 3$, since the asymptotic behaviour is already known for $n = 0,1,2$ by \cite{tsai2009asymptotic}. We will use a construction similar to the proof of Theorem \ref{thm:main}, however we need $k$ and $n$ to be coprime here. \\

\textbf{Step 1}:
Let $\mathcal{S}_n$ be the following set
\[ \mathcal{S}_n := \{ \hspace{1mm}  a \in \mathbb{N} \hspace{1mm} | \hspace{1mm} a \geq \log^2(n) , \hspace{2mm} \gcd(n,a)=1\}. \]
The next lemma shows that $\mathcal{S}_n$ is not `sparse' as a subset of natural numbers.
\begin{lem}
The ratio of any two consecutive members of $\mathcal{S}_n$ is bounded above by $2K$, where $K$ is the constant in Lemma \ref{coprime}. In particular, $K$ is independent of $n$. Moreover, if $a_1$ is the smallest element of $\mathcal{S}_n$ then $a_1 \leq K \log^2(n)$.
\label{consecutive}
\end{lem}
\begin{proof}
Order the elements of $\mathcal{S}_n$ as 
\[ a_1 < a_2 < a_3 < \cdots . \] 
Let $K \geq 2$ be the constant in Lemma \ref{coprime}. Therefore each of the intervals 
\[ [\log^2(n), K \log^2(n)) , [K \log^2(n), 2K \log^2(n)) , [2K \log^2(n), 3K \log^2(n)), \dots \] contains at least one element of $\mathcal{S}_n$. Given $a_i$, either $a_i \in [\log^2(n), K \log^2(n))$ or there is $r \in \mathbb{N}$ such that $a_i \in [r K\log^2(n) , (r+1)K\log^2(n)) $. In the former case, we have 
\[ \frac{a_{i+1}}{a_i} < \frac{2K \log^2(n)}{\log^2(n)}=2K. \]
In the latter case, one has  
\[ \frac{a_{i+1}}{a_i} < \frac{(r+2)K\log^2(n)}{rK \log^2(n)} \leq 3.  \]
Therefore, we may take $2K = \max \{ 2K , 3 \}$ as the desired upper bound.
\end{proof}

\textbf{Step 2:}
From now on we assume that $k \in \mathcal{S}_n$. In particular $k$ and $n$ are coprime. Define the surface $P_{n,k}$ and mapping classes $\rho_1$ and $\rho_2$ as in Step 1 of the proof of Theorem \ref{thm1}. Define the multi curves $\mathcal{B}$ and $\mathcal{R}$ as in Step 2 of the proof of Theorem \ref{thm1}. Let $\eta_{b}$ be the composition of positive Dehn twists along the curves in $\mathcal{B}$. Define $\eta_{r}$ as the composition of negative Dehn twists along the curves in $\mathcal{R}$. 
\begin{notation}
Given $n \in \mathbb{N}$, define $\bar{n} \in \mathbb{Z}$ as follows
\begin{align*}
& n \equiv 1 \hspace{3mm}(\text{mod } 2) \implies \bar{n}=\frac{n-1}{2},  \\
& n \equiv 0 \hspace{3mm}(\text{mod } 4) \implies \bar{n}=\frac{n-2}{2},  \\
& n \equiv 2 \hspace{3mm}(\text{mod } 4) \implies \bar{n}=\frac{n-4}{2}.  
\end{align*}
\end{notation}
Define $\bar{k}$ similarly. The next lemma states that both numbers $\bar{n}$ and $-\bar{n}$ have `large' remainders modulo $n$.
\begin{lem}
Let $n \geq 3$. The number $\bar{n}$ satisfies $\gcd(n,\bar{n})=1$. Moreover, any integer $r $ that is congruent to $\bar{n}$ or $-\bar{n}$ modulo $n$ satisfies $|r|\geq \frac{n}{6}$. 
\label{large-remainder}
\end{lem}

\begin{proof}
We just consider the case $n \equiv 2 \hspace{3mm}(\text{mod } 4)$, as the other two cases are similar. 
\[ n \equiv 2 \hspace{3mm}(\text{mod } 4) \implies \exists q \in \mathbb{Z} \hspace{2mm} \text{s.t.} \hspace{2mm} n = 4q+2 \implies \bar{n} = 2q -1   \]
\[ \implies \gcd(n , \bar{n}) = \gcd(4q+2 , 2q-1)= \gcd(4q+2- 2(2q-1) ,2q-1 )= 1.\]
This proves the first part. The second part follows from 
\[ n = 4q+2 \geq 6 \implies \min \{ \bar{n}, n - \bar{n} \} = \bar{n} \geq \frac{n}{6} \iff q\geq 1. \]
\end{proof}

\begin{definition}
Since $\gcd(n,k)=1$, by the Chinese Remainder Theorem there is a unique number $1 \leq c \leq nk $ such that
\begin{align*}
& c \equiv n^{-1}(\bar{k})^{-1} \hspace{2mm} (\text{mod } k ), \\
& c \equiv k^{-1}(\bar{n})^{-1} \hspace{2mm} (\text{mod } n ).
\end{align*}
Moreover $c$ is coprime to $nk$ by the previous lemma.
\end{definition}

\textbf{Step 3:}
Let $\eta = \tau_{b_1} \circ \tau_{r_1}^{-1}$ and define the map 
\[ f_{n,k} \colon P_{n,k} \longrightarrow P_{n,k}\] 
as 
\[ f_{n,k} := (\rho_2 \circ \rho_1)^c \circ \eta \circ \eta_{b} \circ \eta_{r},\]
where $\eta_r$ and $\eta_b$ are defined as in Step 2. Note we are raising the rotation $\rho_2 \circ \rho_1$ to the power $c$. This is a technical point; the power has been chosen carefully so that the resulting map has a `small' stretch factor.\\

\textbf{Step 4:}
Since $\gcd(k,n)=\gcd(c,kn)=1$, the map $(f_{n,k})^{kn}$ is pseudo-Anosov by Penner's construction, and hence so is $f_{n,k}$. Let $\tau_{n,k}$ be the invariant train track for $f_{n,k}$ obtained from Penner's construction and $V_{\tau}$ be the vector space of transverse measures on $\tau_{n,k}$. For any connected curve $x \in \hat{\mathcal{A}}$, define the transverse measure $\mu_x$ as in Step 4 of the proof of Theorem \ref{thm1}. Let $H$ be the subspace of $V_{\tau}$ spanned by the elements
\[ \{ \mu_x \hspace{1mm} | \hspace{1mm} x \text{ is a connected curve in } \hat{\mathcal{A}} \}. \]
Then $H$ is an invariant subspace under the action of $f_{n,k}$ on $V_\tau$. Denote by $M$ the linear action of $f_{n,k}$ on $H$ in the standard basis. The goal is to give an upper bound for the spectral radius of $M$. 

Let $\Gamma=\Gamma_{n,k}$ be the adjacency graph corresponding to $M$. We partition the vertices of $\Gamma$ into $nk$ sets $U_{i,j}$, where $ i \in \mathbb{Z}_n = \mathbb{Z}/n\mathbb{Z}$ and $ j \in \mathbb{Z}_k = \mathbb{Z} / k\mathbb{Z}$. Define $U_{i,j}$ to be the set of vertices of $\Gamma$ that correspond to the elements 
\[ \{ \mu_y \hspace{1mm} | \hspace{1mm} y \text{ is a connected curve in } \rho_1^i \circ \rho_2^j(\mathcal{A}) \}.\] 
Therefore $U_{i,j}$ form a partition of the vertices of $\Gamma_{n,k}$. The next Lemma follows from the combinatorics of the curves in $\hat{\mathcal{A}}$.

\begin{lem}
If $v \in U_{i,j}$, then $v^+ \subset U_{i+c,j+c}$ except when $|i|+|j|=1$. For these four exceptional sets, the outgoing edges behave as follows
\begin{align*}
 w \in U_{0,1} & \longmapsto w^+ \in U_{c , c+1} \cup U_{c,c}  \\
 w \in U_{0,-1}& \longmapsto w^+ \in U_{c,c-1} \cup U_{c,c}  \\
 w \in U_{1,0} & \longmapsto w^+ \in U_{c+1,c} \cup U_{c,c}  \\
 w \in U_{-1,0}& \longmapsto w^+ \in U_{c-1,c} \cup U_{c,c}.  
\end{align*}
\end{lem}

\begin{definition}
Define the simple directed graph $\bar{\Gamma}= \bar{\Gamma}_{n,k}$ as follows. The vertices consist of  
\[ \{ u_{i,j} \hspace{1mm}| \hspace{1mm} i \in \mathbb{Z}_n, \hspace{2mm} j \in \mathbb{Z}_k \}. \] 
There is an oriented edge from $u_{i,j}$ to $u_{i',j'}$ iff there exists at least one edge in the graph $\Gamma$ from the set $U_{i,j}$ to the set $U_{i',j'}$. 
\end{definition}
In other words, $\bar{\Gamma}$ is the simple directed graph obtained from the graph $\Gamma$ by collapsing each set $U_{i,j}$ to a single vertex $u_{i,j}$. 

For $x \in \mathbb{R}$, let $\lceil x \rceil $ be the smallest integer greater than or equal to $x$.

\begin{lem}
The length of any directed cycle in the graph $\bar{\Gamma}=\bar{\Gamma}_{n,k}$ is strictly larger than $\lceil \frac{nk}{12} \rceil$.
\label{distinct}
\end{lem}

\begin{proof}
Without loss of generality assume that $nk \geq 12$, otherwise the bound is trivial. The proof uses the congruence conditions that was forced upon the number $c$ in its definition. Recall that in the graph $\bar{\Gamma}$, there is a directed edge from each vertex $u_{i,j}$ to the vertex $u_{i+c,j+c}$, where $i \in \mathbb{Z}_n$ and $j \in \mathbb{Z}_k$. Moreover, there are four \emph{exceptional edges}, where an exceptional edge starts at one of  the vertices $u_{i,j}$ for $|i|+|j|=1$ and ends in the vertex $u_{c,c}$. 

Let $\mathcal{C}$ be one of the shortest directed cycles of $\bar{\Gamma}$. Order the vertices of $\mathcal{C}$ as
\[ x_1 \mapsto x_2 \mapsto \dots \mapsto x_{\ell} \mapsto x_1, \]
where $\ell$ is the length of the cycle and $x_i$ are distinct vertices of $\bar{\Gamma}$. First, we observe that there can be at most one exceptional edge in the cycle $\mathcal{C}$. This is because all four exceptional edges in $\bar{\Gamma}$ have the same endpoint, hence if two of them appeared simultaneously in $\mathcal{C}$ then the list of vertices of $\mathcal{C}$ would have had repetition.

Consider two cases:
\begin{enumerate}
\item If there is no exceptional edge in the cycle $\mathcal{C}$: If the index of $x_1$ is equal to $(i_1,j_1)$ (i.e.,  $x_1 = u_{i_1,j_1}$) then the index of $x_s$ is equal to $(i_1+(s-1)c), j_1+(s-1)c)$. Since $x_{\ell+1}=x_1$, their indices should be equal:
\begin{align*}
i_1+\ell c &\equiv i_1 \hspace{3mm} (\text{mod } n), \\
j_1+\ell c &\equiv j_1 \hspace{3mm} (\text{mod } k). 
\end{align*}
Since $\gcd(c,n)=\gcd(c,k)= \gcd(n,k)=1$, the above equations imply that $nk|\ell$. Therefore 
\[  \ell \geq nk > \lceil \frac{nk}{12} \rceil. \]
\item If there is exactly one exceptional edge in the cycle $\mathcal{C}$: The initial vertex of this exceptional edge is one of $u_{0,1}$, $ u_{0,-1}$, $u_{1,0}$ or $u_{-1,0}$. We analyse each case separately. \\
\begin{enumerate}
	\item If the initial vertex is $u_{0,1}$. Recall
	\[ w \in U_{0,1} \longmapsto w^+ \in U_{c , c+1} \cup U_{c,c}.  \]
	Again $x_{\ell+1}=x_1$ implies the following congruence relations
	\begin{align*}
	i_1+\ell c & \equiv i_1 \hspace{3mm} (\text{mod } n), \\
	j_1+\ell c -1 & \equiv j_1 \hspace{3mm} (\text{mod } k).
	\end{align*}
	From the first equation we deduce that $n | \ell$, hence we can write $\ell = n \ell' $ where $\ell'$ is a natural number. Substituting in the second equation, we obtain that
	\begin{align*}
	n c \ell' &\equiv 1 \hspace{3mm} (\text{mod } k) \\
	\implies \ell' \equiv n^{-1}c^{-1} &\equiv \bar{k} \hspace{3mm} (\text{mod } k).
	\end{align*}
	By Lemma \ref{large-remainder} we have $\ell' \geq \frac{k}{6}$, which gives the following bound
	\[ \ell = n \ell' \geq n \cdot \frac{k}{6} \geq \frac{nk}{12}+1 > \lceil \frac{nk}{12} \rceil . \] 
	\item If the initial vertex is $u_{0,-1}$. Proceeding as in the previous case, we get the following system of equations
	\begin{align*}
	\ell c &\equiv 0 \hspace{3mm} (\text{mod } n), \\
	\ell c +1 &\equiv 0 \hspace{3mm} (\text{mod } k).
	\end{align*}
	Setting $\ell = n \ell'$, we deduce that
	\[ \ell' \equiv -(nc)^{-1} \equiv -\bar{k} \hspace{3mm} (\text{mod } k). \]
	Again $\ell \geq \frac{k}{6}$ and the same argument applies.
	\item If the initial vertex is $u_{1,0}$. In this case we obtain the following equations
	\begin{align*}
	\ell c -1 &\equiv 0 \hspace{3mm} (\text{mod } n), \\
	\ell c  &\equiv 0 \hspace{3mm} (\text{mod } k).
	\end{align*}
	From the second equation, we have $k | \ell$. Therefore one can write $\ell = k \ell'$ for a natural number $\ell'$. Substituting in the first equation gives
	\begin{align*}
	kc \ell' &\equiv 1 \hspace{3mm} (\text{mod } n) \\
	\implies \ell' \equiv k^{-1} c^{-1} &\equiv \bar{n} \hspace{3mm} (\text{mod } n).
	\end{align*}
	In particular $\ell' \geq \frac{n}{6}$, which implies that
	\[ \ell = k \ell' \geq  k \cdot \frac{n}{6}>  \lceil \frac{nk}{12} \rceil. \]
	\item If the initial vertex is $u_{-1,0}$. In this case we have $\ell = k \ell'$ such that
	\begin{align*}
	kc \ell' &\equiv -1 \hspace{3mm} (\text{mod } n) \\
	\implies \ell' \equiv -k^{-1} c^{-1} &\equiv -\bar{n}\hspace{3mm} (\text{mod } n).
	\end{align*}
	Again the same argument is carried.
\end{enumerate}
\end{enumerate}

This finishes the proof of the lemma.
\end{proof}

\begin{lem}
There is a constant $E$ such that for any pair of relatively prime numbers $(n, k)$ and for any vertex $v$ in $\Gamma=\Gamma_{n,k}$, the number of directed paths of length $\lceil \frac{nk}{12} \rceil$ and starting from $v$ is at most $E$.
\label{paths}
\end{lem}

\begin{proof}
Observe that most vertices of the graph $\Gamma$ have both ingoing and outgoing degree equal to one. More precisely, given a set $U_{i,j}$, define the ball of radius $1$ around it as 
\[ B_1(U_{i,j}) = U_{i,j} \cup \{ v \hspace{1mm} | \hspace{1mm} v \in U_{i',j'} \text{ and there is an edge between } u_{i,j} \text{ and } u_{i',j'}  \}.  \] 
Let $\mathcal{L}$ be the set of indices $(i,j) \in \mathbb{Z}_n \times \mathbb{Z}_k$ such that, there exists a vertex $v \in B_1(U_{i,j})$ with either $\text{deg}_{\text{in}}(v)  \neq 1$ or $\text{deg}_{\text{out}}(v) \neq 1$. Then the size of $\mathcal{L}$, $|\mathcal{L}| $, is bounded above by a universal constant independent of $n$ and $k$. Define $\mathcal{SL}$ as the set of ordered subsets of $\mathcal{L}$. Therefore, $|\mathcal{SL}|$ is bounded above by a universal constant independent of $n$ and $k$.

Let $\mathcal{P}$ be the set whose elements are directed paths of length $\lceil \frac{nk}{12} \rceil$ and starting from $v$. We show that $|\mathcal{P}|$ is at most $|\mathcal{SL}| \times 12^{\mathcal{L}}$, using the Pigeonhole principle. Assume the contrary that $|\mathcal{P}|>|\mathcal{SL}| \times 12^{\mathcal{L}}$.\\

Let $p \in \mathcal{P}$ be a directed path with vertices
\[ p \colon x_1 \rightarrow x_2 \rightarrow \cdots \rightarrow x_{\ell}.  \]
We look at the vertices $x_t$ that belong to one of the sets $U_{i,j}$ with $(i,j) \in \mathcal{L}$, and record the corresponding indices $(i,j)$ in the order that they appear in $p$. Note the recorded indices are all distinct by Lemma \ref{distinct}. Hence we are assigning an element of $\mathcal{SL}$ to $p$, which we call the \emph{type} of $p$. By the Pigeonhole principle, there is a subset $\mathcal{P}_1 \subset \mathcal{P}$ of size at least $12^{\mathcal{L}}$, all of whose elements have the same type.

Consider a directed path $p \in \mathcal{P}_1$ with vertices  
\[ p \colon x_1 \rightarrow x_2 \rightarrow \cdots \rightarrow x_{\ell}.  \]
We look at the vertices $x_t$ that belong to one of the sets $U_{i,j}$ with $(i,j) \in \mathcal{L}$, and record them in an ordered tuple of length at most $\mathcal{L}$. We call this tuple the \emph{prime location} of $p$. Since the type of elements of $\mathcal{P}_1$ is fixed, and each set $U_{i,j}$ has size $12$, the number of possible locations for elements of $\mathcal{P}_1$ is at most $12^{\mathcal{L}}$. By the Pigeonhole principle, there are two elements $p$ and $q$ of $\mathcal{P}_1$ that have the same prime location. 

Consider the maximal subpath $s$ containing the vertex $v$ that $p $ and $q$ agree on. Let $w$ be the terminal vertex of $s$. Since $p$ and $q$ have the same length and $p \neq q$, we have $s \neq p$ and $s \neq q$. Choose $w_p$ and $w_q$ such that there is an outgoing edge from $w$ to $w_p$ (respectively $w_q$) in the path $p$ (respectively $q$). Assume that $w_p \in U_{i,j}$ and $w_q \in U_{i',j'}$. Consider two cases:
\begin{enumerate}

\item If $(i,j) \not \in \mathcal{L}$ (or $(i',j') \not \in \mathcal{L}$): In this case, any vertex in the ball of radius $1$ around $U_{i,j}$, including $w$, has both ingoing and outgoing degree equal to $1$. However the outgoing degree of $w$ is at least $2$, which is a contradiction.

\item If $(i,j) \in \mathcal{L}$ and $(i',j') \in \mathcal{L}$: Since $p$ and $q$ have the same type, we must have $(i,j) = (i',j')$. As $p$ and $q$ have the same prime location, we have $w_p = w_q$, which is a contradiction. 

\end{enumerate}
The contradiction shows that the claimed bound holds, and the proof is complete.

\end{proof}

\textbf{Step 5:}
\begin{lem}
Denote the stretch factor of $f_{n,k}$ by $\lambda_{n,k}$. There is a universal constant $E$ independent of $k$ and $n$ such that
\[\log(\lambda_{n,k})  \leq \frac{72\log(E)}{g_{n,k}}, \]
where $g_{n,k}= (6k-1)n+1$ is the genus of $P_{n,k}$. 
\end{lem}

\begin{proof}
Denote the spectral radius of $M$ by $\rho(M)$. Lemma \ref{paths} implies that
\begin{align*}
\rho(M)^{\frac{nk}{12}} &\leq E \\
\implies (\frac{nk}{12}) \log(\rho(M)) &\leq \log(E).
\end{align*}
Note that $\rho(M) = \lambda_{n,k}$. Therefore
\[\log(\lambda_{n,k}) = \log(\rho(M)) \leq \frac{12\log(E)}{nk} \leq \frac{72\log(E)}{g_{n,k}}, \]
since $g_{n,k}= (6k-1)n+1 \leq 6kn$.
\end{proof}

\textbf{Step 6:}
Denote the mapping torus of $f_{n,k}$ by $M_{n,k}$. Let $\mathcal{C}_{n,k} \subset H_2(M_{n,k}; \mathbb{Z})$ be the fibered face of the Thurston norm corresponding to the monodromy $f_{n,k}$. There exists a norm-minimizing surface $F_{n,k}$ of genus two in $M_{n,k}$ such that $F_{n,k} \subset \overline{\mathcal{C}}_{n,k}$. This can be shown, as in the proof of Lemma \ref{genustwosurface}, by following the image of $\gamma$ under iterations of $f_{n,k}$. 

For $r \geq 0$ define the homology class $[P_{n,k}^r] \in \mathcal{C}_{n,k}$ as 
\[ [P_{n,k}^r] = [P_{n,k}]+ r [F_{n,k}]. \]
An embedded surface $P_{n,k}^r$ representing the homology class $[P_{n,k}^r]$ is obtained by taking the oriented sum of $P_{n,k}$ and $r$ copies of $F_{n,k}$. Denote the genus of $P_{n,k}^r$ by $g_{n,k}^r$. Using the linearity of the Thurston norm on a face of the Thurston norm, we can compute (see the proof of Lemma \ref{cover})
\[ g_{n,k}^r= g_{n,k}+r .\]
The surface $P_{n,k}^r$ is the fiber of a fibration of $M_{n,k}$ with pseudo-Anosov monodromy $f_{n,k}^r$.

\begin{lem}
If $k <k'$ are two consecutive members of $\mathcal{S}_n$, then 
\[ \frac{g_{n,k'}}{g_{n,k}} \leq 4K,  \]
where $K$ is the constant in Lemma \ref{coprime}. In particular, this ratio is bounded above by a constant independent of $k$ and $n$. 
\end{lem}

\begin{proof}
\[ \frac{g_{n,k'}}{g_{n,k}} = \frac{(6k'-1)n+1}{(6k-1)n+1} \leq \frac{(6k'-1)n}{(6k-1)n}= \frac{6k'-1}{6k-1} \leq 2 \frac{k'}{k} \leq 4K.   \]
Here the first inequality follows from 
\[ u \geq v>0 \implies \frac{u+1}{v+1} \leq \frac{u}{v}.\]
The second inequality follows from direct computation, and the third inequality holds by Lemma \ref{consecutive}.
\end{proof}

\textbf{Step 7:} 
Let $f_{n,k}^r$ be the pseudo-Anosov map defined as in Step 6. Denote the stretch factor of $f_{n,k}^r$ by $\lambda_{n,k}^r$.
\begin{lem}
Let $k< k'$ be two consecutive members of $\mathcal{S}_n$, and $K$ be the constant in Lemma \ref{coprime}. There is a universal constant $A$ independent of $n$, $ k$ and $r$ such that, as $r$ varies between $0$ and $(4K-1)g_{n,k}$, the genera of the surfaces $P_{n,k}^r$ cover all natural numbers between $g_{n,k}$ and $g_{n,k'}$, and 
\[  \log(\lambda_{n,k}^r) \leq \frac{A}{g_{n,k}^r}. \]
Moreover if $k =a_1$ is the smallest element in $\mathcal{S}_n$, then 
\[ g_{n,a_1} \leq 6K \cdot n \log^2(n).  \]
\end{lem}

\begin{proof}
Recall that $g_{n,k}^r = g_{n,k} +r$ by Step 6. Moreover, we have 
\[ \frac{g_{n,k'}}{g_{n,k}} \leq 4K.  \]
This proves the first part of the lemma. Let $h\colon \mathcal{C}_{n,k} \rightarrow \mathbb{R}$ be the Fried's function as in Lemma \ref{Fried}, which extends the entropy function on primitive integral points. We have
\[ h([P_{n,k}^r])<h([P_{n,k}])=\log(\lambda_{n,k}) \leq  \frac{72 \log(E)}{g_{n,k}} \leq 4K \cdot \frac{72 \log(E)}{g_{n,k}^r}.  \]
Here the first inequality is by Proposition \ref{agol}, using the identity $[P_{n,k}^r] = [P_{n,k}]+ r[F_{n,k}]$ and $[F_{n,k}] \in \overline{\mathcal{C}}_{n,k}$. The second inequality is by Step 5, and the third inequality follows from 
\[ g_{n,k}^r = g_{n,k}+r \leq g_{n,k}+ (4K-1)g_{n,k}= 4K \cdot g_{n,k}. \]
Hence we can take $A = 4K \cdot 72 \log(E)$, where $K$ and $E$ are the constants in Lemma \ref{coprime} and Lemma \ref{paths} respectively.
For the last part, note that 
\[ g_{n,a_1} = (6a_1 -1)n+1 \leq 6 a_1 n \leq 6K \cdot n \log^2(n), \]
since $a_1 \leq K \log^2(n)$ by Lemma \ref{consecutive}.
\end{proof}

\textbf{Final Step:}
Let $n \geq 3$. Define $\mathcal{S}_n$ as in Step 1, and let $k \in \mathcal{S}_n$. Define the pseudo-Anosov map $f_{n,k} \colon P_{n,k} \rightarrow P_{n,k}$ as in Step 3, and denote the genus of $P_{n,k}$ by $g_{n,k}$. Let $K$ be the constant in Lemma \ref{coprime}. Define pseudo-Anosov maps $f_{n,k}^r  \colon P_{n,k}^r \rightarrow P_{n,k}^r$ for $0 \leq r \leq  (4K-1)g_{n,k}$ as in Step 6. By Step 7, there is a universal constant $A$ independent of $n$, $ k$ and $r$ such that, as $r$ varies between $0$ and $(4K-1)g_{n,k}$, the genera of the surfaces $P_{n,k}^r$ cover all natural numbers between $g_{n,k}$ and $g_{n,k'}$, and 
\[  \log(\lambda_{n,k}^r) \leq \frac{A}{g_{n,k}^r}. \]
Moreover, if $k =a_1$ is the smallest element in $\mathcal{S}_n$ then 
\[ g_{n,a_1} \leq 6K \cdot n \log^2(n).  \]
This proves the upper bound, with $C= 6K$.\\

The lower bound follows directly from Penner's lower bound. Note that for $n \geq 3$ we have
\[ n \leq \frac{g}{C \log^2(n)} <  \frac{g}{C}. \]
This implies 
\[\  l_{g,n} \geq \frac{\log(2)}{12g-12+4n} \geq \frac{\log(2)}{12g -12+ 4( \frac{g}{C}) } \geq \frac{B}{g},  \]
where $B = \frac{\log(2)}{12 + \frac{4}{C}}$.\\
\end{proof}

\begin{remark}
One can prove a similar but slightly weaker result without using Iwaniec's theorem. In this case, we define $\mathcal{S}_n$ as 
\[ \mathcal{S}_n := \{ \hspace{1mm}  a \in \mathbb{N} \hspace{1mm} | \hspace{1mm} a \geq n , \hspace{2mm} \gcd(n,a)=1\}. \]
Since any interval of length $n$ contains a number that is coprime to $n$, the ratio of any two consecutive elements of $\mathcal{S}_n$ is bounded above by $3$. Moreover if $a_1$ is the smallest element of $\mathcal{S}_n$ then $g_{n,a_1} \leq 6n^2$. Following the above proof, we obtain a similar result for $g \geq 6n^2$, and with universal constants $A$ and $ B$ that are constructive.
\end{remark}
We finish by mentioning that we do not know the answer to the following special case of the main question yet.
\begin{question}
Let $0<r_1 <r_2$ be constants. Determine the behaviour of $l_{g,n}$ on the two-dimensional cone $r_1  n < g< r_2  n$ as a function of two variables $g$ and $n$.
\end{question}

\bibliographystyle{plain}
\bibliography{references}

\vspace{4mm}

\noindent Mathematical Institute, University of Oxford, Oxford, UK \\
Email address: \textit{yazdi@maths.ox.ac.uk}

\end{document}